\documentclass[12pt]{amsart}

\usepackage{mathrsfs}
\usepackage{amsmath,amssymb,enumerate,tikz}
\usepackage{amsfonts,amssymb}
\usepackage{pinlabel}
\usepackage{latexsym,amsfonts,amssymb,verbatim,mathrsfs,amsthm}
\usepackage{amsmath,amsthm,amssymb,latexsym,graphics,textcomp,graphicx,subfigure}
\usepackage{eucal,eufrak}
\usepackage{colonequals}
\usepackage{graphicx}
\usepackage{url}

\theoremstyle{plain}
\newtheorem{theorem}{Theorem}[section]
\newtheorem{lemma}[theorem]{Lemma}
\newtheorem{corollary}[theorem]{Corollary}

\newtheorem{definition}[theorem]{Definition}

\newtheorem{proposition}[theorem]{Proposition}

\theoremstyle{definition}
\newtheorem{remark}{Remark}
\newtheorem{question}{Question}

\newcommand{\ML}{\mathcal{ML}(S)}

\newcommand{\Co}{\mathcal{C}_0(S)}
\newcommand{\T}{\mathcal{T}}

\newcommand{\E}{\textup{Ext}}
\newcommand{\Sim}{\mathcal{S}(S)}
\newcommand{\Cim}{\mathcal{C}(S)}
\newcommand{\Arc}{\mathcal{A}(S)}
\newcommand{\Bnd}{\mathcal{B}(S)}

\title{Almost isometries between Teichm\"uller spaces }

\date{\today}
\author{ Manman Jiang\and Lixin Liu \and  Huiping Pan}

\address{ Manman Jiang\\
Guangzhou Maritime University, 510275, Guangzhou, P. R. China}
\email{jiangmanm@126.com}
\address{Lixin Liu\\
School of Mathematics and Computational Science, Sun Yat-Sen University, 510275, Guangzhou, P. R. China}
\email{mscllx@mail.sysu.edu.cn}
\address{Huiping Pan\\
 School of Mathematical Science, Fudan University, 200433, Shanghai, P. R. China}
\email{panhp@fudan.edu.cn}

\thanks{This work is partially supported by NSFC, No: 11271378.}

\begin{document}

\begin{abstract}
We prove that the Teichm\"uller space of surfaces with given boundary lengths equipped with the arc metric (resp. the Teichm\"uller metric) is almost isometric to the Teichm\"uller space of punctured surfaces equipped with the Thurston metric (resp. the Teichm\"uller metric).
\end{abstract}
\maketitle


\vskip 20pt
\section{introduction}
Let $S$ be an oriented surface of genus $g$ with $n$ boundary components such that $n\geq 1$.  The Euler characteristic of $S$ is $\chi(S)=2-2g-n$.  Throughout this paper we assume that $\chi(S)\leq 0$.
 Recall that a \textit{marked complex structure} on $S$ is a pair $(X,f)$ where $X$ is a Riemann surface and $f:S\to X$ is an orientation preserving  homeomorphism. Two marked complex structures $(X,f)$ and $(Y,h)$ are called \textit{equivalent} if there is a conformal map homotopic to $f\circ h^{-1}$.  Denote by $[X,f]$ the equivalence class of $(X,f)$. The set of equivalence classes of marked complex structures is the Teichm\"uller space  denoted by $T_{g,n}$.


 Let $X$ be a Riemann  surface with boundary.  There exist two different hyperbolic metrics on $X$. One is of infinite area obtained from the Uniformization theorem, the other one is of finite area obtained from the restriction to $X$ of the hyperbolic metric on its (Sckottky) double such that  each boundary component is a smooth simple closed geodesic (see \S \ref{ssec:double}). The second one is called the intrinsic metric on  $X$. In this paper when we mention a hyperbolic metric on a surface with nonempty boundary we mean  the second one.
 The correspondence between complex structure and hyperbolic metric provides another approach for the Teichm\"uller theory.
 Recall that a \textit{ marked hyperbolic surface} $(X, f)$ is a hyperbolic surface $X$  equipped with an orientation-preserving homeomorphism $f : S\rightarrow X$, where $f$ maps each component of the boundary of  $S$ to a geodesic boundary of $X$. Two marked hyperbolic surfaces  $(X,f)$ and $(Y,h)$ are called \textit{equivalent} if there is an isometry  homotopic to $f\circ{h}^{-1}$ relative to the boundary.
The  Teichm\"{u}ller space  $T_{g,n}$  is also the set of equivalence classes of marked hyperbolic surface. For simplicity, we will denote a point $[X,f]$ in ${T}_{g,n}$ by $X$, without explicit reference to the marking or to the equivalence relation.

 Let $\beta_1,\cdots,\beta_n$ be the boundary components of $S$.  For any $\Lambda=(\lambda_1,...,\lambda_n)$ $ \in R_{\geq0}^n$.   Let  $T_{g,n}(\Lambda)\subset T_{g,n}$ be the set of the equivalence classes of marked hyperbolic metrics  whose boundary components have hyperbolic lengths $(l(\beta_1),...,l(\beta_n))=\Lambda$. In particular, $T_{g,n}(0)$ is the Teichm\"uller space of surfaces with $n$ punctures.  
 It is clear that $T_{g,n}=\cup_{\Lambda\in \mathbb R^n_+}T_{g,n}(\Lambda)$.  Let $\Gamma=\{\gamma_1,\cdots,\gamma_{3g-3+n}\}$ be a pants decomposition of $S$, i.e. the complement of $\Gamma$  on $S$ consists of $2g-2+n$ pairs of pants $\{R_i\}_{i=1}^{2g-2+n}$.  Let $\mu$ be a set of disjoint simple closed curves whose restriction to any pair of pants $R_i$ consists  of three arcs, such that any two of the arcs are not free homotopic with respect to the boundary of $R_i$. The pair $(\Gamma,\mu)$ is called a \textit{marking} of $S$. For any $X\in T_{g,n}$, let $(\mathcal L,\mathcal T,\Lambda)$ be the corresponding Fenchel-Nielsen coordinates with respect to the marking $(\Gamma,\mu)$, where $\mathcal L=(l_1,\cdots,l_{3g-3+n})$ represents the lengths of $\{\gamma_1,\cdots,\gamma_{3g-3+n}\}$, $\mathcal T=(t_1,\cdots,t_{3g-3+n})$ represents the twists along $\{\gamma_1,\cdots,\gamma_{3g-3+n}\}$  and $\Lambda=(\lambda_1,...,\lambda_n)$ represents the lengths of the boundary components (for details about Fenchel-Nielsen coordinates we refer to \cite{Bu}). The Fenchel-Nielsen coordinates induce a  natural homeomorphism between Teichm\"uller spaces $T_{g,n}(\Lambda)$ and  $T_{g,n}(0)$ in the following way:
 \begin{eqnarray*}
   \Phi_\Gamma: T_{g,n}(\Lambda)&\longrightarrow& T_{g,n}(0)\\
   (\mathcal L,\mathcal T,\Lambda)&\longmapsto&(\mathcal L,\mathcal T,0).
 \end{eqnarray*}

The goal of this paper is to compare various metrics on the Teichm\"uller spaces $T_{g,n}(\Lambda)$ and  $T_{g,n}(0)$ via the homeomorphism $\Phi_{\Gamma}$.
\begin{definition}\label{def}
  Two metric spaces $(X_1,d_1)$ and $(X_2,d_2)$ are called \textit{almost isometric} if there exist a map $f:X_1\to X_2$, two positive constants $A$ and $B$, such that both of  the following two conditions hold.
\begin{enumerate}
  \item For any $x,y\in X_1$,
      \[ |d_2(f(x),f(y))-d_1(x,y)|\leq B.\]
  \item For any $z\in X_2$, there exists $x\in X_1$ such that
     \[ d_2(z,f(x))\leq A. \]
\end{enumerate}
\end{definition}
\subsection{The Thuston metric and the arc metric}

 An \textit{essential simple closed curve} on $S$ is a simple closed curve which is not homotopic to a single point or a boundary component. An \textit{essential arc}  is a simple  arc whose endpoints  are on the boundary and which is  not homotopic to any subarc of the boundary. Let ${\Sim} $ be the set of homotopy classes of essential simple closed curves on S,  ${\Arc}$ be the set of homotopy classes of essential arcs on S, and $\Bnd$ be the set of  homotopy classes of the boundary components.

For any $X_1,X_2\in T_{g,n}(\Lambda)$, define
$$d_{Th}(X_1,X_2):=\log \sup_{[\alpha]\in\Sim}\frac{l_{X_2}([\alpha])}{l_{X_1}([\alpha])}$$
and
$$d_{A}(X_1,X_2):=\log\sup_{[\alpha]\in \Arc}\frac{l_{X_2}([\alpha])}{l_{X_1}([\alpha])}.$$

From the works \cite{Pan} and \cite{LPST2}, both $d_{Th}$ and $d_A$ are asymmetric metric on $T_{g,n}(\Lambda)$, which are called the \textit{Thurston metric} and the \textit{arc metric} respectively.  Moreover, the authors (\cite{LPST2})  observed that
$$d_{A}(X_1,X_2)=\log\sup_{[\alpha]\in \Arc\cup\Bnd\cup\Sim}\frac{l_{X_2}([\alpha])}{l_{X_1}([\alpha])}.$$

Our first result is the following.
\begin{theorem}\label{thm:arc-almost}
  $(T_{g,n}(\Lambda),d_A)$ and  $(T_{g,n}(0),d_{Th})$ are almost isometric. More precisely, there is a constant $C_1$ depending on the surface $S$ and boundary lengths $\Lambda$ such that,
  \[ |d_A(X_1,X_2)-d_{Th}(\Phi_\Gamma(X_1),\Phi_\Gamma(X_2))|\leq C_1.\]
\end{theorem}
\remark  Papadopoulos-Su (\cite{PS}) considered the case where $\Lambda$ is close to zero, they showed that the constant $C_1$ in Theorem \ref{thm:arc-almost} will tend to zero if $\Lambda$ tends to zero.

\begin{proof}[Proof of Theorem \ref{thm:arc-almost}]
   To prove Theorem \ref{thm:arc-almost}, it suffices to verify that they satisfy the two conditions in Definition \ref{def}. The first condition follows from Theorem \ref{thm:arc-thu} and Theorem \ref{thm:thu-almost}. The second condition follows from the fact that $\Phi_\Gamma$ is a homeomorphism.
 \end{proof}
\begin{theorem} \label{thm:arc-thu}
The arc metric and the Thurston metric are almost-isometric in $T_{g,n}(\Lambda)$. More precisely, there is a constant $C_2$ depending on the surfaces $S$ and boundary lengths $\Lambda$ such that,
  \[ |d_A(X_1,X_2)-d_{Th}(X_1,X_2)|\leq C_2.\]
\end{theorem}
\remark Liu-Papadopoulos-Su-Th\'eret (\cite[Theorem 3.7]{LPST2})
proved that $ d_{Th}$ and $ d_A$ are almost isometric on the $\epsilon_0$-relative $\epsilon-$  thick part of the Teichm\"uller space of surfaces with boundary. Later,  Liu-Su-Zhong (\cite[Theorem 1.5]{LSZ}) proved that  the symmetrizations  of these two metrics  $ d_{L}(X_1,X_2):=\max\{d_{Th}(X_1,X_2),d_{Th}(X_2,X_1)\}$ and  $ d_{AL}(X_1,X_2):=\max\{d_{A}(X_1,X_2),d_{A}(X_2,X_1)\}$ are almost isometric on the $\epsilon$ thick part of the Teichm\"uller space of surfaces with boundary.
\begin{theorem}[\cite{Pan}]\label{thm:thu-almost}
  $(T_{g,n}(\Lambda),d_{Th})$ and  $(T_{g,n}(0),d_{Th})$ are almost isometric. More precisely, there is a constant $C_3$ depending on the surfaces $S$ and boundary lengths $\Lambda$ such that,
  \[ |d_{Th}(X_1,X_2)-d_{Th}(\Phi_\Gamma(X_1),\Phi_\Gamma(X_2))|\leq C_3.\]
\end{theorem}

\subsection{The Teichm\"uller metric}
The arc metric and the  Thurston metric describe the deformation of hyperbolic metric on the surface, while the Teichm\"uller metric describes the deformation of conformal structure ( complex structure). Given two marked complex structures $[X_1, f_1]$ and $[X_2,f_2]$, the Teichm\"uller metric is defined by
\[ d_T([X_1,f_1],[X_2,f_2])=\frac{1}{2}\log \inf\{K(f): f \textup{ is isotopic to } f_1\circ (f_2)^{-1}\},\]
where $K(f)$ represents the quasiconformal dilation of $f$.

For closed surfaces, Kerckhoff expressed  the Teichm\"uller metric in terms of the extremal length of simple closed curves in the following way. For any $X_1,X_2$ in the Teichm\"uller space,
\begin{equation}\label{eq:Kerchoff}
  d_T(X_1,X_2):=\frac{1}{2}\sup_{[\alpha]}\log\frac{\E_{X_2}([\alpha])}{\E_{X_1}([\alpha])},
\end{equation}
where the sup ranges over all  essential simple closed curves on the surface.

For surfaces with boundary, Liu-Papadopoulos-Su-Th\'eret (\cite{LPST}) developed similar result. They described the Teichm\"uller metric in terms of the extremal lengths of  essential arcs and boundary components.

The theorem below is our second main result.
\begin{theorem}\label{thm:FN}
  For small $\epsilon$, $(\T_{g,n}(\epsilon),d_T)$ and $( \T_{g,n}(0),d_T)$ are  almost isometric. More precisely,  
  for any $X, Y\in \T_{g,b}(\epsilon)$,
  $$ |d_T(X_1,X_2)-d_T(\Phi_\Gamma(X_1),\Phi_\Gamma(X_2)) |\leq \log (n+3).$$
\end{theorem}
\remark The constant $\log (n+3)$ is not optimal.

The organization of this paper is as following. In \S 2, we  recall some basic concepts and facts. In \S 3, we  prove Theorem \ref{thm:arc-thu}. In \S4, we prove Theorem \ref{thm:FN}. Finally, we collect a few questions in \S 5.

\section{Preliminary}
\subsection{(Sckottky) double  and Teichm\"uller map}\label{ssec:double}
Let  $X$ be a Riemann surface with nonempty boundary. We can represent $X$ as $H/G$, where $H$ is the upper half plane and $G$ is a torsion-free Fuchsian group of second kind. There is an infinite set $\mathcal I$ of open intervals $I$ on the extended real axis $\mathbb R\cup \infty$ such that $G$ acts properly discontinuously on $H\cup J\cup L$ where $L$ is the lower half plane and $J$ is the union of all $I\in\mathcal I$. $X^d:=(H\cup J\cup L)/G$ is called the (Schottky) \textit{double} of $X$ and $\bar X:=L/G$ is called the \textit{mirror image} of $X$. The restriction of the hyperbolic metric on $X^d$ to $X$ is called the \textit{intrinsic metric} on $X$.  It is clear that in the intrinsic metric each boundary component is a simple closed geodesic. The double of any essential geodesic arc on $X$  is a simple closed geodesic on $X^d$.

An \textit{admissible quadratic differential} on $X$ is the restriction to $X$  of a holomorphic quadratic differential $q^d$ on $X^d$  such that
\begin{enumerate}
  \item at each puncture on $X$,  $q^d$ has at worst a first order pole.
  \item $\partial X $ is an $q^d$-horizontal line.
\end{enumerate}
Note that the symmetry requires that the zeroes of $q^d$ on $\partial X$ have even order.  Away from the zeroes of $q^d$, there is a local coordinate
$ \zeta=\xi+i\eta$
such that $q^d=d\zeta^2$. Let $ \bar\zeta=K\xi+i\eta$, $0<K<\infty$ . $\bar \zeta$ defines a new Riemann surface $\bar X^d$. The map $f:\zeta \mapsto \bar \zeta$ is called the \textit{Teichm\"uller map} from $X^d$ to $\bar X^d$ with initial quadratic differential $q^d$. The restriction of $f$ to $X$ is called  the \textit{Teichm\"uller map} from $X$ to $\bar X$ with initial quadratic differential $q$.
 Given two marked Riemann surfaces $[X_1,f_1],[X_2,f_2]$ with boundary, there is a unique Teichm\"uller map $f:X_1\to X_2$ homotopic to $f_2\circ f_1^{-1}$ minimizing the quasiconformal dilation (\cite{Abikoff}) such that the initial quadratic differential is an admissible quadratic differential on $X$.
\subsection{Measured lamination}
Given a hyperbolic surface $X$ with nonempty geodesic boundary, a simple geodesic   is  one of the  four types below:
\begin{itemize}
  \item an essential simple closed geodesic;
  \item a geodesic boundary component;
  \item an essential geodesic arc;
  \item  an infinite  geodesic  in the interior.
\end{itemize}
A \textit{geodesic lamination} $m$ on $X$ is a closed subset of $X$ consisting of mutually disjoint simple geodesics which  are called leaves of this geodesic lamination. A \textit{transverse invariant measure} $\mu$ of a geodesic lamination $m$ is a Radon measure defined on every
arc $k$ transverse to the support of $m$ such that  $\mu$ is invariant with respect to any
homotopy of $k $ relative to the leaves of $m$. A \textit{ measured geodesic lamination} is a lamination
$m$ endowed with a transverse invariant measure $\mu$. For simplicity, we denote by $\mu$ the measured geodesic lamination $(m,\mu)$. Each measured geodesic lamination $\mu$ induces a functional $i_\mu$ over $\Sim\cup\Bnd$ in the following way:
\begin{eqnarray*}
  i_\mu:\Sim\cup\Bnd&\longrightarrow &\mathbb R_{\geq0}\\
  \ [\alpha]&\longmapsto& i(\mu,[\alpha]):=\inf_{\alpha'\in[\alpha]}\int_{\alpha'} d\mu.
\end{eqnarray*}
Two measured geodesic laminations $\mu,\mu'$  are said to be equivalent if $i(\mu,[\alpha])=i(\mu',[\alpha])$ for any $[\alpha]\in\Sim\cup\Bnd$. Denote by $\mathcal{ML}(X)$ the space of equivalence classes of measured geodesic laminations  on $X$ equipped with the topology that $\mu_n$ converges to $\mu$ if for any $[\alpha]\in\Sim\cup\Bnd$, $i(\mu_n,[\alpha])$ converges to $i(\mu,[\alpha])$.   Since there is a one-to-one correspondence between $\mathcal{ML}(X)$ and $\mathcal{ML}(X')$ for two different hyperbolic metrics $X$ and $X'$, we denote by $\ML$ the space of equivalence classes of measured geodesic lamination without pointing to any specific hyperbolic metric. Hubbard and Masur (\cite{HM}) proved that there is a homeomorphism between $\ML$ and the space of the horizontal measured  foliations of admissible quadratic differentials on $X$.  For surfaces of finite type, $\Sim\times \mathbb R_+$ is dense in $\ML$ in this topology. But for the surfaces with boundary, this is no longer true. The simplest counterexample is an essential geodesic arc since it is not in the closure of $\Sim\times \mathbb R_+$. Let $\mathcal{ML}_0(\mathcal S)$ be a subset of $\ML$ consisting of measured foliations whose leaves are either essential simple closed geodesics or infinite geodesics in the interior. It is  clear  that $\Sim\times \mathbb R_+$ is dense in $\mathcal{ML}_0(\mathcal S)$.

\subsection{Extremal length}
Let $\alpha$ be a simple closed curve or an essential arc, and $X$ be a Riemann surface. A \textit{conformal metric} on $X$ is a metric which can be expressed as $\rho(z)|dz|$ locally. The extremal length of $\alpha$ on $X$ is defined by:
\begin{equation}
  \E_X(\alpha):=\sup_{\rho}\frac{l^2_\rho(\alpha)}{Area(\rho)},
\end{equation}
where the sup ranges over all the conformal metrics on $X$, $Area(\rho)$ is the area of $X$ endowed with the metric $\rho$, and $l_\rho(\alpha):=\inf_{\alpha'\in[\alpha]}\int _{\alpha'}\rho |dz|$. It is clear that $$l^2_{a\rho}(\alpha)/Area(a\rho)=l^2_{\rho}(\alpha)/Area(\rho)$$
 for any positive constant $a$. There exist a unique conformal metric up to scaling realizing the supremum which is called the \textit{extremal metric} (see \cite{Strebel}).
The extremal length is a conformal invariant. For surfaces without boundary, Kerckhoff extended the definition of extremal length from $\Sim\times \mathbb R_+$ to $\ML$.  For surfaces with  boundary, this extension also holds by considering the double $X^d$ of $X$.

The following lemmas will be used in this paper.
\begin{lemma}\label{lem:extremal}
  Suppose $X\in T_{g,n}(S)$. Let $\mu=\mu_1+\mu_2+\cdots+\mu_k$ be a measured geodesic lamination where $\mu_j\in\mathbb R^+\times (\Sim\cup\Bnd)$, $j=1,2,\cdots,k$. Then
  \[\max_{1\leq j\leq k}\{\E_X(\mu_j)\}\leq\E_X(\mu)\leq k^2\max_{1\leq j\leq k}\{\E_X(\mu_j)\}.\]
\end{lemma}
\begin{proof}
  Let $\rho_i, \rho_\mu$ be the extremal metrics of $\mu_i$ and $\mu$ respectively such that $Area(\rho_i)=Area(\rho_\mu)=1$, $i=1,2,\cdots,k$. Without loss of generality, we assume that
  $\E_X(\mu_1)=\max_{1\leq j\leq k}\{\E_X(\mu_j)\}$. Then
  \[\E_X(\mu)\geq \frac{l^2_{\rho_1}(\mu)}{Area(\rho_1)}\geq \frac{l^2_{\rho_1}(\mu_1)}{Area(\rho_1)}=\E_X(\mu_1).\]
  On the other hand,
  \begin{eqnarray*}
    \E_X(\mu)&=& l^2_{\rho_\mu}(\mu)\\
    &=&(l_{\rho_\mu}(\mu_1)+l_{\rho_\mu}(\mu_2)+\cdots+l_{\rho_\mu}(\mu_k))^2\\
    &\leq&(\sqrt{\E_X(\mu_1)}+\sqrt{\E_X(\mu_2)}+\cdots+\sqrt{\E_X(\mu_k)})^2\\
    &\leq&k^2\max_{1\leq j\leq k}\{\E_X(\mu_j)\}.
  \end{eqnarray*}
\end{proof}

\begin{lemma}[Maskit, \cite{Maskit}]\label{Maskit}
  Let $Y\in\T_{g,n}(0)$ and $\alpha$ be a nontrival simple closed curve , then
  \begin{enumerate}
    \item $l_X(\alpha)$ and $\E_X(\alpha)$ goes to zero together, and  $$\lim_{l_X(\alpha)\to0}l_X(\alpha)/\E_X(\alpha)=\pi.$$
    \item  $ \frac{l_X(\alpha)}{\pi}\leq \E_X(\alpha)\leq \frac{l_X(\alpha)}{2}e^{l_X(\alpha)/2}. $
  \end{enumerate}
\end{lemma}
\remark The statements above also holds for $X\in\T_{g,n}(\Lambda)$ with $\Lambda\in\mathbb R_+^n$.  In fact, suppose $X\in T_{g,n}(\Lambda)$, let $X^d$ be the double of $X$. Let $\alpha$ be a simple closed curve on $X$ and $\alpha^d$ be its double on $X^d$. Then $l_X(\alpha)=l_{X^d}(\alpha)/2$ and $\E_X(\alpha)=\E_{X^d}(\alpha)/2$.
\subsection{Fenchel-Nielsen coordinates}\label{ssect:FN}

 \begin{figure}[b]
  \subfigure[]
  {
  \includegraphics[width=50mm]{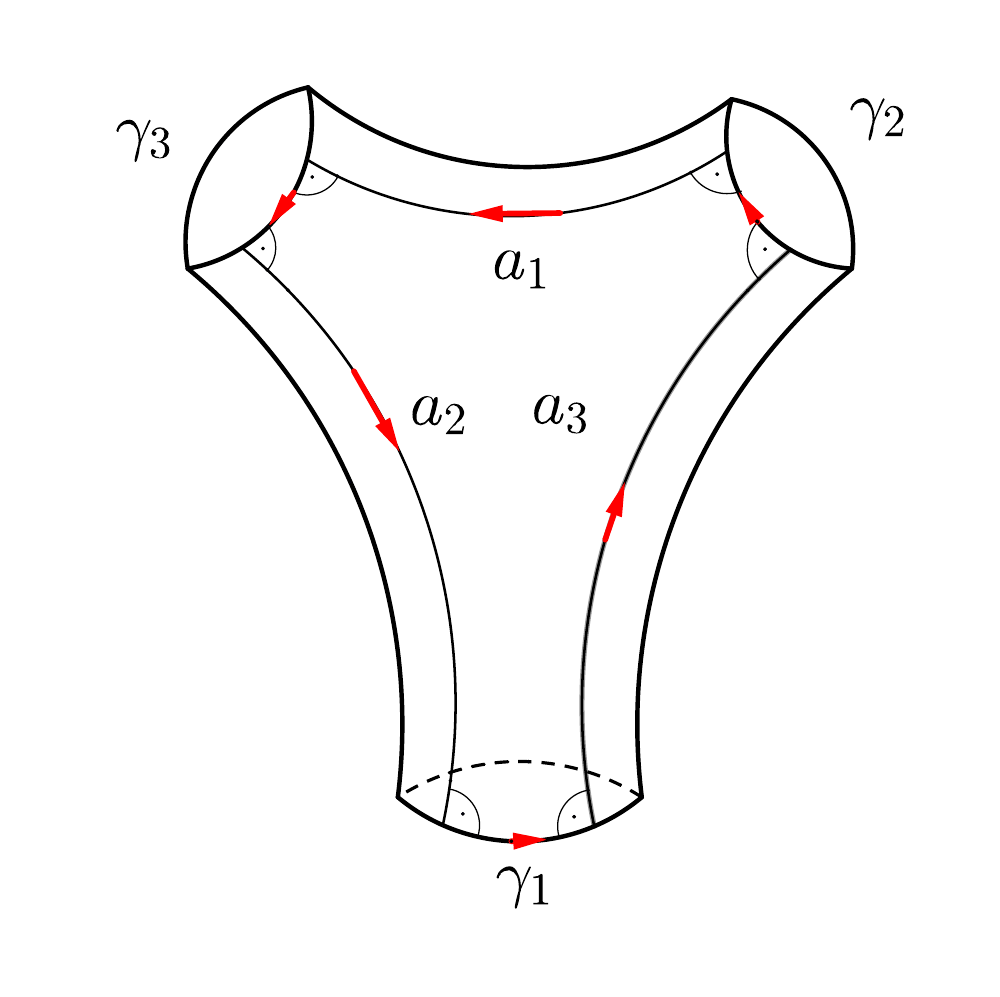}
  \label{fig:Y}
  }
    \subfigure[]
  {
  \includegraphics[width=50mm]{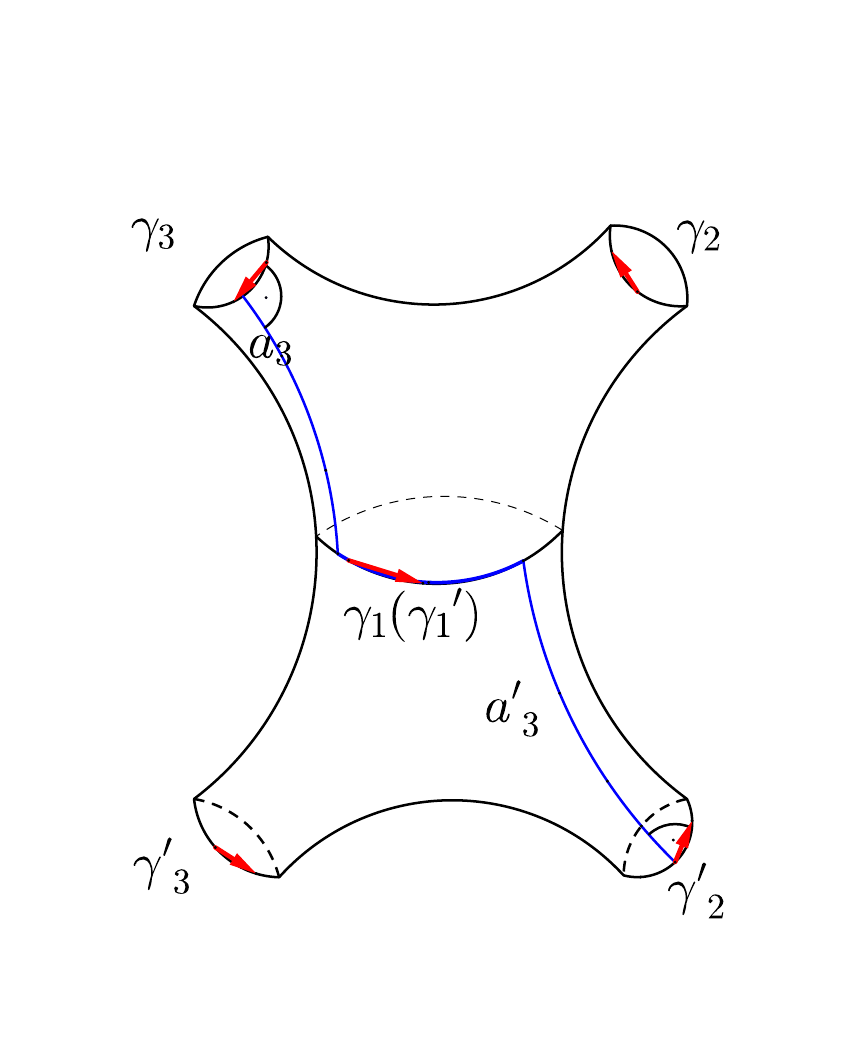}
  \label{fig:Y-Yt}
  }
  \caption{}
  \end{figure}
 Let $R$ be a pair of pants with boundaries $\{\gamma_1,\gamma_2,\gamma_3\}$. Let $a_1,a_2, a_3$ be three geodesic arcs orthogonal to the boundaries (see Figure \ref{fig:Y}). Choose an orientation for each boundary  such that $R$ is on the left. Let $\gamma_i:[0,1]\to \gamma_i$ be a parametrization of $\gamma_i$ with constant speed such that $\gamma_1([0,1/2]),$ $a_3,\gamma_2([0,1/2])$, $a_1$, $\gamma_3([0,1/2])$, $ a_3$ consist a hexagon. We call this parametrization a \textit{stand parametrization}. A homeomorphism $f$ between two pairs of pants $R,R'$ is called \textit{boundary coherent} if $f\circ \gamma_i(s)=\gamma_i'(s)$ for $s\in [0,1]$ and $i=1,2,3$.

  For two pairs of pants $R,R'$,  if $l(\gamma_i)=l(\gamma')$ for some $i=1,2,3$, we can  paste $R$ and $R'$  along $\gamma,\gamma'$ in the following way (see Figure \ref{fig:Y-Yt})
 \[ \gamma_i(s)=\gamma'_i(t-s)\]
 for some $t\in\mathbb R$. We say $R$ and $R'$ are pasted along $\gamma_i$ with twist $t$.

Now we give an explaination for the Fenchel-Nielsen coordinates $(\mathcal L,\mathcal T, \Lambda)$. Let $\Gamma=\{\gamma_1,\cdots,\gamma_{3g-3+n}\}$ be a pants decomposition of $X$ and $\{R_1,\cdots,R_{2g-2+n}\}$ be the corresponding $2g-2+n$ pairs of pants with stand parametrization. $\mathcal L$ and $\Lambda$ determine  these $2g-2+n$ pairs of pants,  and $\mathcal T$ tells us how to paste these pairs of pants.

\section{proof of Theorem \ref{thm:arc-thu}}
Let $\gamma\in A(S)$ be an arc connecting the boundary components  $\beta_i$ and $\beta_j$ ($\beta_i$ may equal to $\beta_j$). Then for any hyperbolic structure $X$, there is a unique geodesic in the relative homotopy class of $\gamma$, which is orthogonal to $\beta_i,\beta_j$ at each endpoint.  We still denote it by $\gamma$. It is not hard to see that a tubular neighborhood of $\beta_i\cup\beta_j\cup\gamma$ is a topological pair of pants. Let us call this pants determined by $\gamma$.

\

\begin{proof}(proof of Theorem \ref{thm:arc-thu})
It follows from the definitions that
\[d_{Th}(X_1,X_2)\leq d_A(X_1,X_2) \] for any
$X_1,X_2\in T_{g,n}$.

To control the arc metric from above by the Thurston metric, it suffices to find an essential simple closed curve $\alpha$ for each essential arc $\gamma \in \mathcal A'(S):=\{\gamma\in\mathcal A(S): l_{X_2}(\gamma)>l_{X_1}(\gamma)\}$ ($\alpha$ depends on $\gamma$) such that
$$\frac{l_{X_2}(\alpha)}{l_{X_1}(\alpha)}\geq C\frac{l_{X_2}(\gamma)}{l_{X_1}(\gamma)}$$
for some constant $C$ which depends on the surface $S$ and the boundary lengths $\Lambda$.
We discuss for the two cases.

 \textbf{Case(1)}:  $\gamma$ connects  two different boundary components $\beta_i,\beta_j$, see Figure \ref{fig:Y1}.
 Then there is another simple closed cure $\alpha\in \Sim$, such that $\beta_i,\beta_j,\alpha$ are the boundaries of the pants determined by $\gamma$.

  \begin{figure}[b]
  \subfigure[]
  {
  \includegraphics[width=50mm]{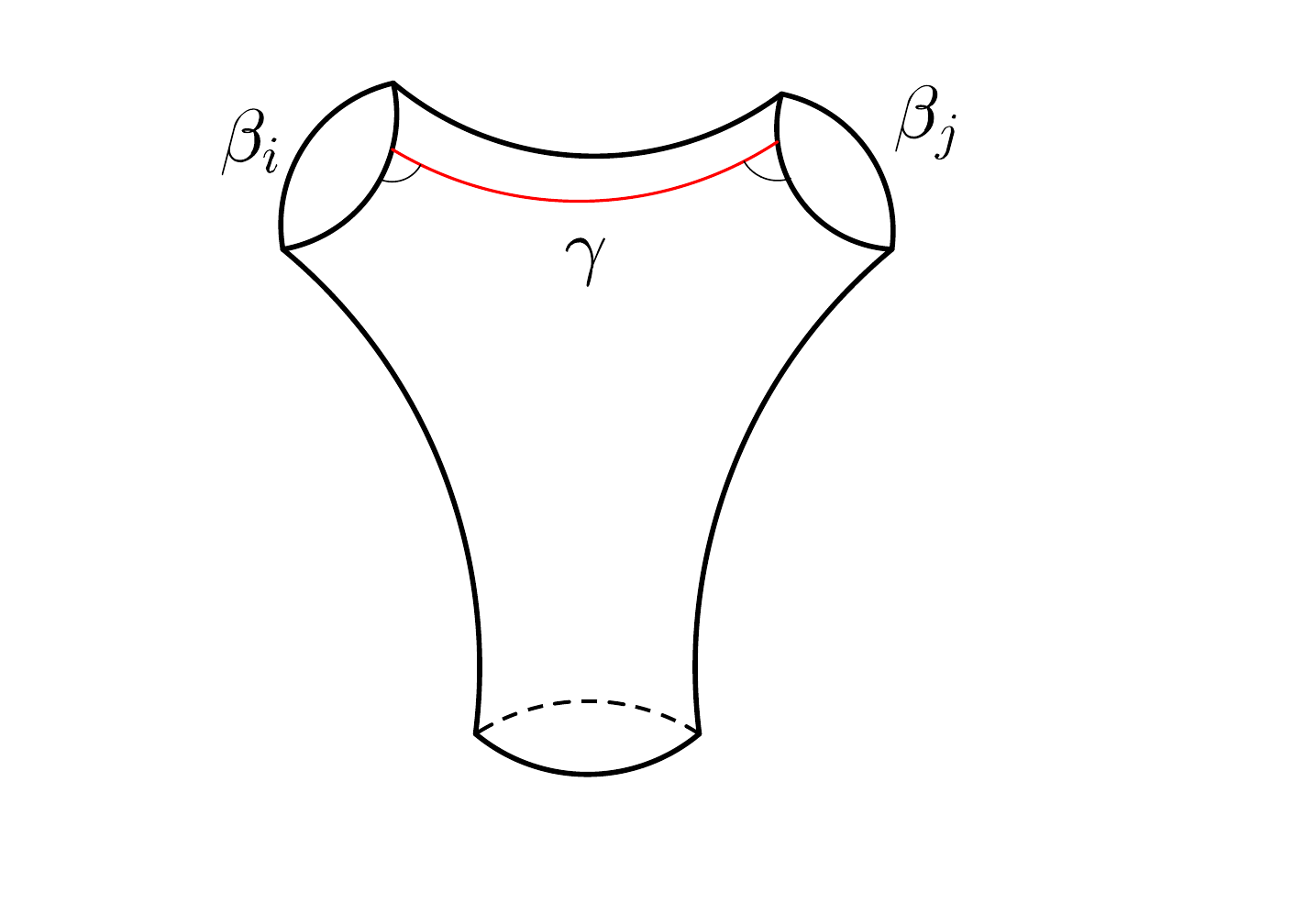}
  \label{fig:Y1}
  }
    \subfigure[]
  {
  \includegraphics[width=50mm]{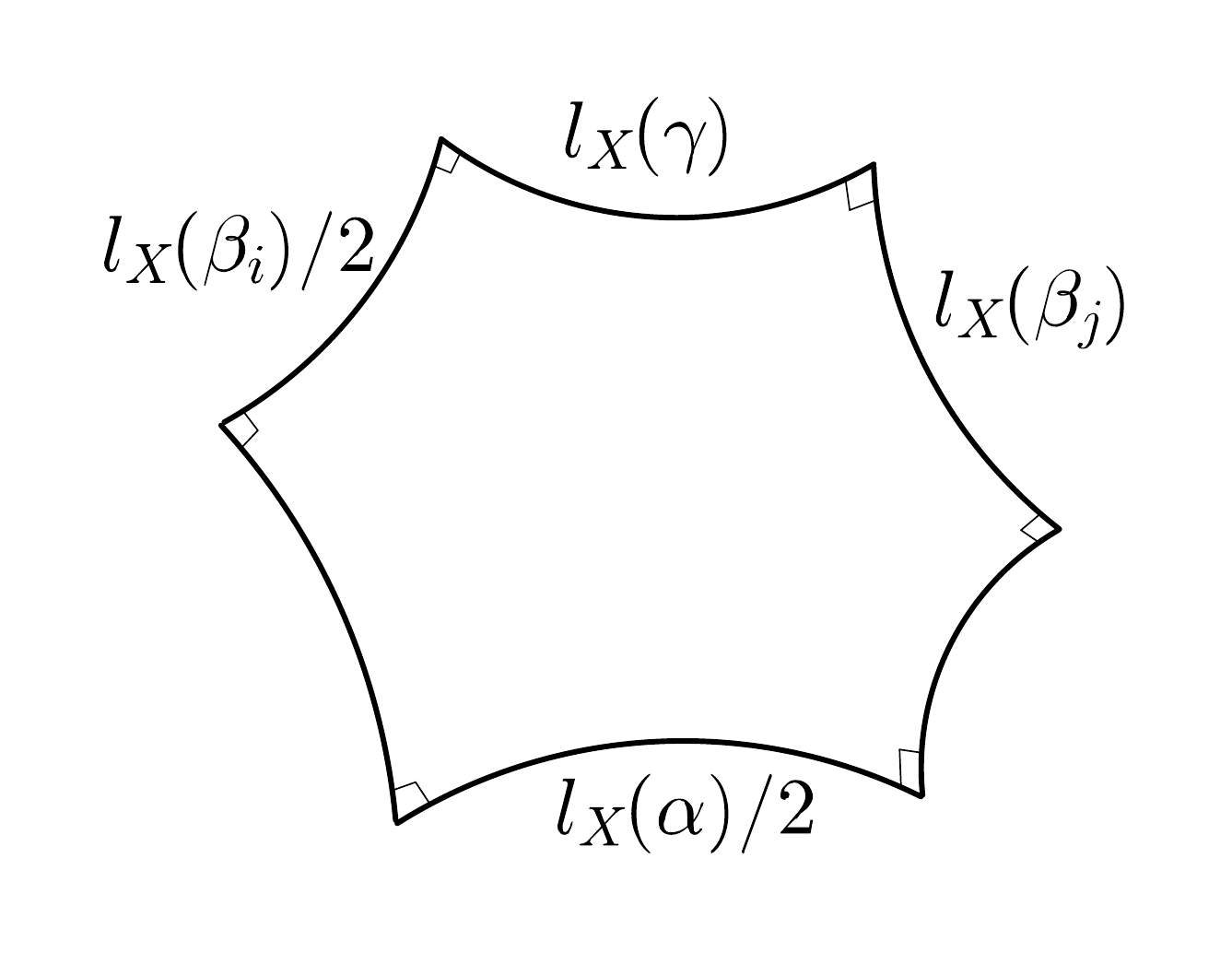}
  \label{fig:Hexagon1}
  }
  \caption{}
\end{figure}

  For any $X\in T_{g,n}(L)$ with  $l_X({\beta_i})=\lambda_i,l_X{(\beta_j)}=\lambda_j$,  we have
  \begin{equation}\label{eq:gamm}
  \cosh(l_{X}(\gamma))=\frac{\cosh{\frac{\lambda_i}{2}}\cosh{\frac{\lambda_j}{2}+
  \cosh{\frac{l_X(\alpha)}{2}}}}{\sinh{\frac{\lambda_i}{2}}\sinh{\frac{\lambda_j}{2}}}.
  \end{equation}
Let $\lambda=\max_{1\leq i,j \leq n}\{\sinh{\frac{\lambda_i}{2}}\sinh{\frac{\lambda_j}{2}},
\frac{\cosh{\frac{\lambda_i}{2}}\cosh{\frac{\lambda_j}{2}}
+1}{\sinh{\frac{\lambda_i}{2}}\sinh{\frac{\lambda_j}{2}}}\}.$

Therefore $$\frac{e^{\frac{l_X(\alpha)}{2}}}{2\lambda}\leq \cosh(l_{X}(\gamma))\leq \lambda e^{\frac{l_X(\alpha)}{2}} .$$

On the other hand,
 $$\frac{e^{l_{X}(\gamma)}}{2}\leq \cosh(l_{X}(\gamma))\leq e^{l_{X}(\gamma)},$$

So we get
$$-2\log2\lambda\leq l_X(\alpha)-{2l_X(\gamma)}\leq2\log2\lambda.$$

Let $K=\log2\lambda$.
\begin{itemize}
  \item If $l_X(\gamma)\geq K,$ then
$$1\leq\frac{l_X(\alpha)}{l_X(\gamma)}\leq3.$$
  \item If $l_X(\gamma)\leq K,$
set $r_0=\min_{1\leq i,j\leq n}\cosh^{-1}(\frac{\cosh{\frac{\lambda_i}{2}}\cosh{\frac{\lambda_j}{2}}}{\sinh{\frac{\lambda_i}{2}}\sinh{\frac{\lambda_j}{2}}}),$
it follows from (\ref{eq:gamm}) that
$$l_X(\gamma)\geq r_0,$$
and

$$l_X(\alpha)\leq x_0:=2\max_{1\leq i,j\leq n} \cosh^{-1}(K{\sinh{\frac{\lambda_i}{2}}\sinh{\frac{\lambda_j}{2}}}-{\cosh{\frac{\lambda_i}{2}}\cosh{\frac{\lambda_j}{2}}}).$$
\end{itemize}

 For any $\gamma \in \mathcal A'(S):=\{\gamma\in\mathcal A(S): l_{X_2}(\gamma)>l_{X_1}(\gamma)\}$,  we consider the following situations.
\begin{itemize}
\item If $l_{X_2}(\gamma)\geq l_{X_1}(\gamma)\geq K,$ then
$$\frac{l_{X_2}(\alpha)}{l_{X_1}(\alpha)}\geq \frac{l_{X_2}(\gamma)}{3l_{X_1}(\gamma)}.$$

\item If $l_{X_2}(\gamma)\geq K\geq l_{X_1}(\gamma),$ then  $$\frac{l_{X_2}(\alpha)}{l_{X_1}(\alpha)}\geq\frac{l_{X_2}(\gamma)}{l_{X_1}(\gamma)}
\frac{l_{X_1}(\gamma)}{x_0}\geq\frac{l_{X_2}(\gamma)}{l_{X_1}(\gamma)}\frac{r_0}{x_0}.$$

\item If $ K\geq l_{X_2}(\gamma)\geq l_{X_1}(\gamma),$ then
$$\frac{l_{X_2}(\alpha)}{l_{X_1}(\alpha)}\geq 1\geq\frac{l_{X_2}(\gamma)}{l_{X_1}(\gamma)}\frac{r_0}{K}.$$

\end{itemize}

Let $C_1=\max\{\frac{1}{3},\frac{r_0}{x_0},\frac{r_o}{K}\}$,  we  have
$$\log\frac{l_{X_2}(\alpha)}{l_{X_1}(\alpha)}\geq C_1\frac{l_{X_2}(\gamma)}{l_{X_1}(\gamma)}.$$

\vskip 20pt
\textbf{Case(2)}: $\gamma$ connects  $\beta_i$ to itself for some boundary component $\beta_i$, see Figure \ref{fig:Y2}. Then there exist another two simple closed curves $\alpha,\delta \in \Cim$, such that $\beta_i,\alpha,\delta$ are the boundaries of the pants determined by $\gamma$.
  \begin{figure}
  \subfigure[]
  {
  \includegraphics[width=50mm]{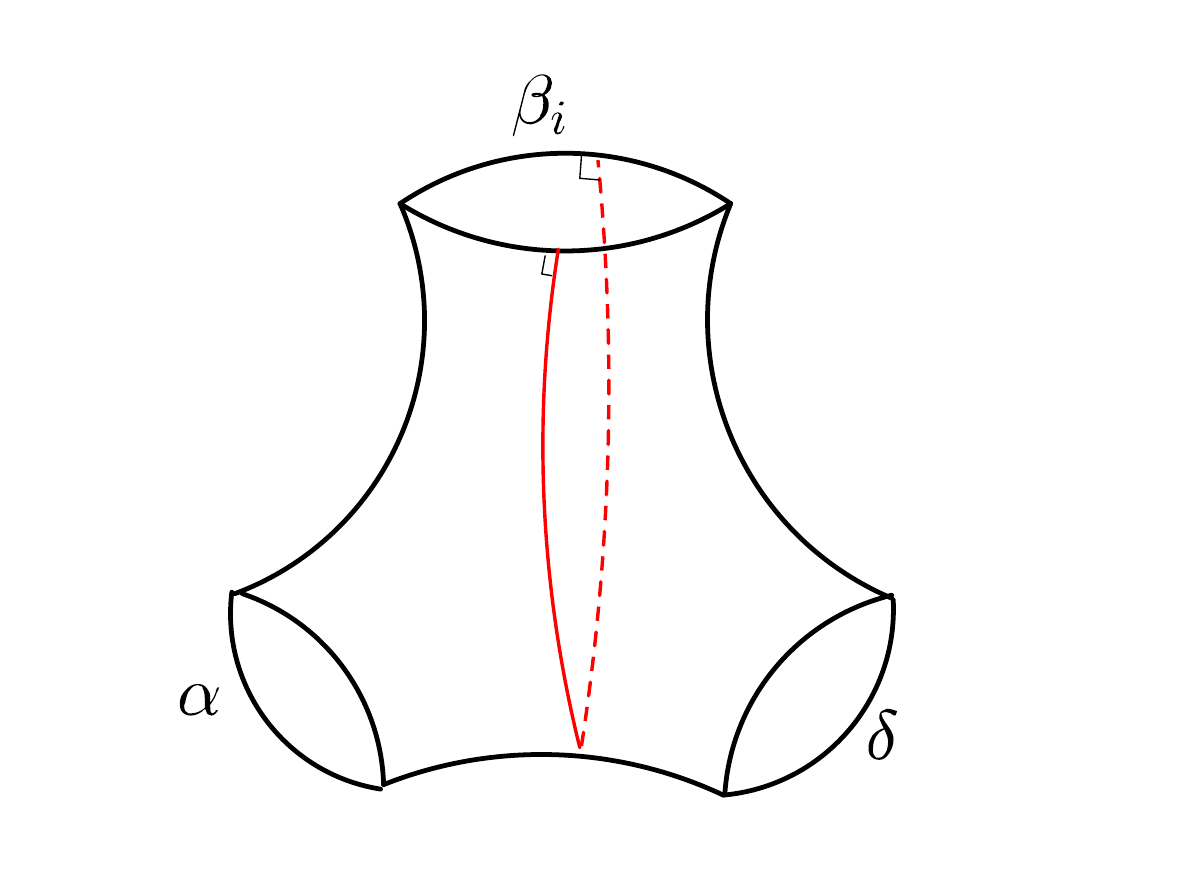}
  \label{fig:Y2}
  }
    \subfigure[]
  {
  \includegraphics[width=50mm]{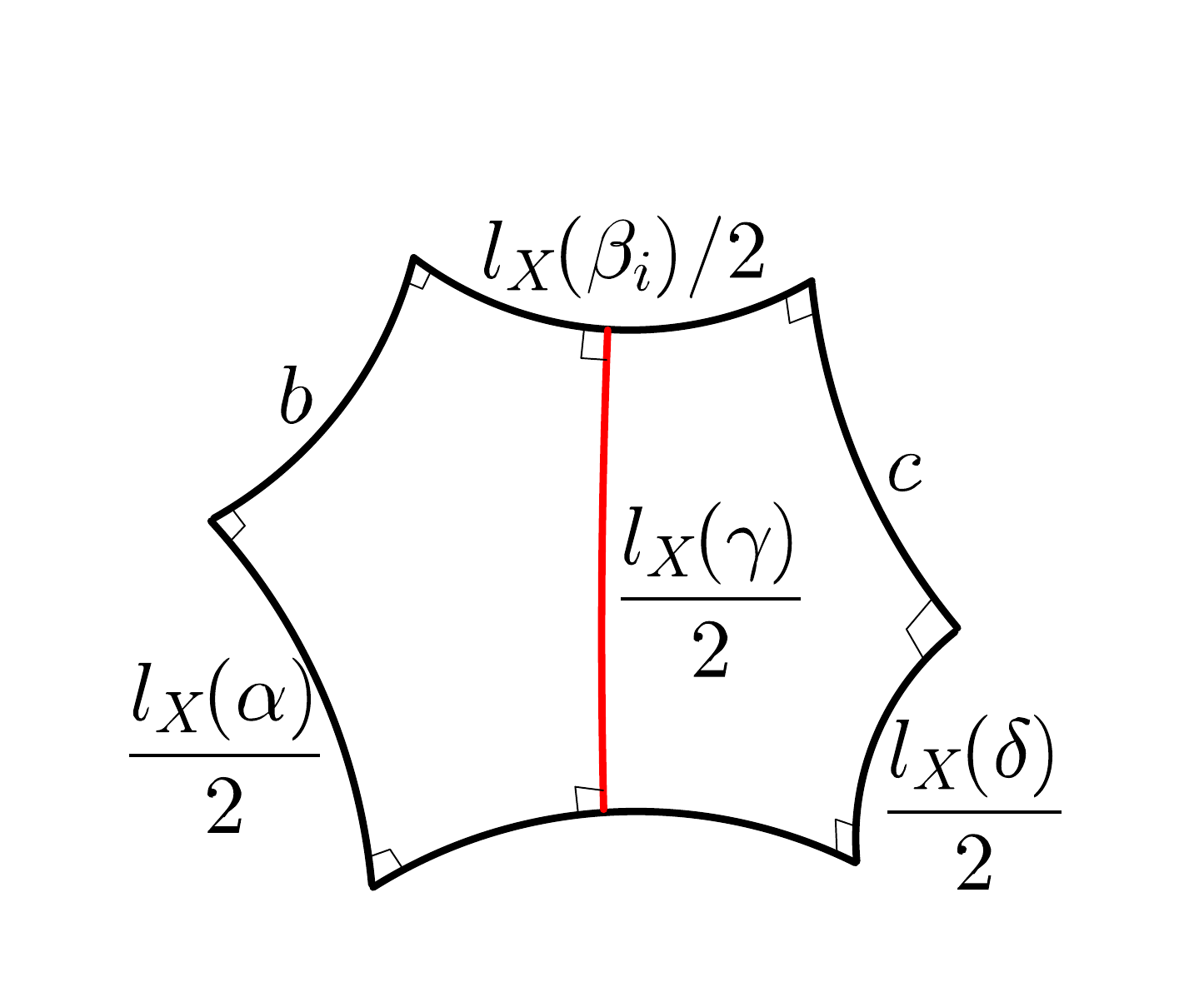}
  \label{fig:Hexagon2}
  }
  \caption{}
\end{figure}
Hence
\begin{eqnarray*}
  \cosh^2(\frac{l_X(\gamma)}{2})&=& \sinh^2\frac{l_X(\alpha)}{2}\sinh^2 b\\
 & =& \sinh^2\frac{l_X(\alpha)}{2}(\cosh^2b-1)\\
&=&\sinh^2\frac{l_X(\alpha)}{2}[(\frac{\cosh{\frac{l_X(\alpha)}{2}}\cosh{\frac{\lambda_i}{2}
+\cosh{\frac{l_X(\delta)}{2}}}}{\sinh{\frac{l_X(\alpha)}{2}}\sinh{\frac{\lambda_i}{2}}})^2-1]\\
&=&\frac{[\cosh\frac{l_X(\delta)}{2}+\cosh(\frac{l_X(\alpha)}{2}+\frac{\lambda_i}{2})]
[\cosh\frac{l_X(\delta)}{2}+\cosh(\frac{l_X(\alpha)}{2}-\frac{\lambda_i}{2})]}{\sinh^2\frac{\lambda_i}{2}}.
\end{eqnarray*}

On the other hand,
$$ e^{\frac{-\lambda_i}{2}}\cosh\frac{l_X(\alpha)}{2}\leq\cosh(\frac{l_X(\alpha)}{2}\pm\frac{\lambda_i}{2})\leq e^{\frac{\lambda_i}{2}}\cosh\frac{l_X(\alpha)}{2}.$$

From the discussions above, we get
$$ \frac{e^{\frac{-\lambda_i}{2}}}{\sinh\frac{\lambda_i}{2}}\leq \frac{\cosh\frac{(l_X(\gamma))}{2}}{(\cosh\frac{l_X(\delta)}{2}+
\cosh\frac{l_X(\alpha)}{2})}\leq\frac{e^{\frac{\lambda_i}{2}}}{\sinh\frac{\lambda_i}{2}}.$$

Then

$$ \log \frac{e^{\frac{-\lambda_i}{2}}}{2\sinh\frac{\lambda_i}{2}}\leq l_X(\gamma)-\max\{{l_X(\delta)},{l_X(\alpha)}\}\leq\log \frac{4e^{\frac{\lambda_i}{2}}}{\sinh\frac{\lambda_i}{2}}.$$
Note that
$$\cosh\frac{l_X(\alpha)}{2}\leq \sinh\frac{l_X(\beta_i)}{2}\sinh\frac{l_X(\gamma)}{2}.$$
Therefore
$$l_X(\gamma)\geq2\sinh^{-1}\frac{1}{\sinh\frac{l_X(\beta_i)}{2}}
=2\sinh^{-1}\frac{1}{\sinh\frac{\lambda_i}{2}}.$$

 The same as the discussion in Case 1, we have
$$\max\{\frac{l_{X_2}(\alpha)}{l_{X_1}(\alpha)},\frac{l_{X_2}(\delta)}{l_{X_1}(\delta)}\}\geq \frac{\max\{l_{X_2}(\delta),l_{X_2}(\alpha)\}}{\max\{l_{X_1}(\delta),l_{X_1}(\alpha)\}}
\geq C_2 \frac{l_{X_2}(\gamma)}{l_{X_1}(\gamma)}$$
for some constant $C_2$.



\vskip 20pt
Combing Case(1) and Case(2) we know that for any arc $\gamma\in\mathcal A'(S)$, we can find a simple closed curve $\alpha'\in \Sim$, such that
$$\frac{l_{X_2}(\alpha')}{l_{X_1}(\alpha')}\geq C\frac{l_{X_2}(\gamma)}{l_{X_1}(\gamma)},$$
where $C=\max\{C_1,C_2\}.$

Consequently,
 $$d_{A}(X_1,X_2)-\log C\leq d_{Th}(X_1,X_2)\leq d_{A}(X_1,X_2).$$

\end{proof}

\section{Proof of Theorem \ref{thm:FN}}
As we mentioned in the introduction, Liu-Papadopoulos-Su-Th\'eret (\cite{LPST}) described the Teichm\"uller metric on  the Teichm\"uller space of surfaces with boundary via the extremal lengths of essential arcs and the boundary components. Follow the idea in \cite{LPST}, we get the following approximation.

\begin{proposition}\label{prop:teich}
For small $\epsilon$,  $X, Y\in \T_{g,n}(\epsilon)$, we have
  $$|d_T(X,Y)-\frac{1}{2}\sup_{[\alpha]\in \Sim}|\log\frac{\E_{Y}(\alpha)}{\E_{X}(\alpha)}||\leq \log(n+2).$$
\end{proposition}
\begin{proof}
  Let  $X^d, Y^d$ be the double of $X$, $Y$ respectively. It follows from  (\ref{eq:Kerchoff})  that
  \begin{eqnarray*}
      d_T(X,Y)&=&d_T(X^d,Y^d)\\
      &=&\frac{1}{2}\sup_{\alpha\in \mathcal{ML}(X^d)}|\log\frac{\E_{Y^d}(\alpha)}{\E_{X^d}(\alpha)}|\\
      &\geq&\frac{1}{2}\sup_{\alpha\in \Sim\cup \Bnd}|\log\frac{\E_{Y}(\alpha)}{\E_{X}(\alpha)}|.
  \end{eqnarray*}

  It remains to prove the other direction.
  Let $q^d$ be the initial quadratic differential  associated to the  Teichm\"uller map between $X^d$ and $Y^d$. Let $h^d$  be the horizontal measured lamination of $q^d$. Then
   \[d_T(X,Y)=d_T(X^d,Y^d)=\frac{1}{2}|\log\frac{\E_{Y^d}(h^d) }{\E_{X^d}(h^d)}|.\]

   Recall that  $\partial X$ is a $q^d$- horizontal line (see \S\ref{ssec:double}), $h^d$ can be decomposed as $h^d=\Sigma_{i=1}^n a_i\beta_i+\mu$, where $a_i\geq0$, $\beta_i$ is a measured lamination represented by a  boundary component of $X$, and $\mu$ is an interior measured geodesic lamination, i.e.  $a_i\beta_i\in \mathbb R_{\geq0}\times\Bnd$, $\mu\in\mathcal{ML}_0(\mathcal S)$. Since $\mathbb R^+\times\Sim$  is dense in  $\mathcal{ML}_0(\mathcal S)$, there exists $\{c_i\delta_i\}_{i=1}^\infty\subset\mathbb R^+\times\Sim$ such that $c_i\delta_i\to\mu$ as $i\to\infty$. Therefore,
   \[d_T(X,Y)=\lim_{j\to\infty}\frac{1}{2}|\log\frac{\E_{Y^d}(\Sigma_{i=1}^n a_i\beta_i+c_j\delta_j) }{\E_{X^d}(\Sigma_{i=1}^n a_i\beta_i+c_j\delta_j)}|.\]

 For each  $j\geq1$, let $\mu_j\in\{a_1\beta_1,a_2\beta_2,\cdots, a_n\beta_n, c_j\delta_j\}$ such that $$\E_{Y^d}(\mu_j)=\max_{1\leq k\leq n}\{\E_{Y^d}(a_k\beta_k),\E_{Y^d}(c_j\delta_j)\}.$$
 It follows from Lemma \ref{lem:extremal} that
    \begin{eqnarray*}
      d_T(X,Y)&\leq&\frac{1}{2}\lim_{j\to\infty}|\log\frac{(n+1)^2\E_{Y}(\mu_j) }{\E_{X}(\mu_j)}|\\
      &\leq &  \frac{1}{2}\sup_{\alpha\in \Sim}|\log\frac{\E_{Y}(\alpha)}{\E_{X}(\alpha)}|+\frac{1}{2}\sup_{\beta\in \Bnd}|\log  \frac{\E_{Y}(\beta) }{\E_{X}(\beta)}|+\log (n+1)\\
      &\leq &  \frac{1}{2}\sup_{\alpha\in \Sim}|\log\frac{\E_{Y}(\alpha)}{\E_{X}(\alpha)}|+\log (n+2)
    \end{eqnarray*}
    where we use the first result of Lemma \ref{Maskit}  in the last inequality.
\end{proof}

Next, we estimate the extremal lengths.
Let $X\in T_{g,n}(0)$,  $\{p_1,\cdots,p_n\}$  be the punctures of $X$. It is well known that every puncture  has a cusp neighbourhood consisting of horocycles of length less than $1$ (see \cite{Bu} for example). Let $\{D^\epsilon_1,\cdots,D^\epsilon_n\}$ be  the corresponding cuspidal neighborhoods with boundary lengths $\epsilon$,  $\textup{Cusp}_\epsilon(X):=\cup_{1\leq i\leq n}D^\epsilon_i$ and
$X_\epsilon:=X\backslash\textup{Cusp}_\epsilon(X)$. The following proposition is key to prove Theorem \ref{thm:FN}.
\begin{proposition}\label{prop:ext-est-1}
  For small $\epsilon$, there is a constant $C_\epsilon$ such that for any $\alpha\in \Sim$ and any $X\in\T_{g,n}(0)$,
  $$ {1}\leq \frac{\E_{X_\epsilon}(\alpha)}{\E_{X}(\alpha)}\leq {C_\epsilon}.$$
   Moreover, $C_\epsilon\to1$ as $\epsilon\to0$.
\end{proposition}

\begin{proof}

  Before we prove the proposition, we make some conventions.  For any simple closed curve $\alpha$, denote by $L_\rho(\alpha)$ the length of $\alpha$ under the metric $\rho$ and  $l_\rho(\alpha)$  the length of the geodesic representative of $\alpha$ under the metric $\rho$.

  Since $X_\epsilon\subset X$, it follows from the definition that ${\E_{X_\epsilon}(\alpha)}\geq{\E_{X}(\alpha)}.$

  Let $\{p_1,\cdots,p_n\}$  be the punctures of $X$ and $\{D^\epsilon_1,\cdots,D^\epsilon_n\}$ be  the corresponding cuspidal neighborhoods with boundary lengths $\epsilon$. Recall that each puncture $p_i$ has a cuspidal neighbourhood $D^1_i$ with boundary length $1$ such that $D^1_i\cap D^1_j=\emptyset$ for $i\neq j$. Let $G_i$ be the infinite cyclic group generated by a simple closed curve which is homotopic to $p_i$. Let $D^*=\{w:0<|w|<1\}$ be the punctured unit disc equipped with the hyperbolic metric $\rho=|dw|/(|w|\log|w|^{-1})$. Let $\pi_i:D^*\mapsto X$ be a covering map such that   the fundamental group of $D^*$  corresponds to $G_i$ and that  $X$ coincides with the push-forward of $\rho$. In this setting, $D^\epsilon_i$ is conformal to the punctured disc $D_{R(\epsilon)}^*=\{w:0<|w|<{R(\epsilon)}\}$  where $R(\epsilon)=\exp(-2\pi/\epsilon)$. It is clear that $R(\epsilon)<1/2R(1)$ for small $\epsilon$.

  The remaining of the proof will be split into two cases.

  \textbf{Case 1}: $\E_{X_\epsilon}(\alpha)\leq \sqrt \epsilon.$ It is clear that $\E_{X}(\alpha)\leq\E_{X_\epsilon}(\alpha)\leq \sqrt \epsilon.$
  Let $\phi$ be the quadratic differential on $X$ whose horizontal measured foliation $h_\phi$ is equivalent to $\alpha$.     Denote by $|\phi|$ the induced flat metric on $X$, then $|\phi|$ is the extremal metric of $\alpha$, i.e.
 \begin{equation}\label{eq:ext}
  \E_X(\alpha)=\frac{l^2_{|\phi|}(\alpha)}{||\phi||},
  \end{equation}
  where $||\phi||=\int_X |\phi|$ and $l_{|\phi|}(\alpha)$ is the length of the geodesic homotopic to $\alpha$ under $|\phi|$.

  To estimate $\E_{X_\epsilon}(\alpha)$, we need to estimate the length of $\partial D^\epsilon_i$, denoted by $L_{|\phi|}(\partial D_i^\epsilon)$, under the flat metric $|\phi|$.  Recall that $\phi$ has a simple pole at $p_i$, it has the following expression in $D^*_{R(1)}$,
  \[\phi(w)dw^2=(\psi(w)/w)dw^2,\]
  where $\psi$ is holomorphic. Note that $|\psi(w)|$ is subharmonic and $\int_0^{2\pi}|\psi(re^i\theta)|d\theta$ is an increasing function of $r$.
  For simplicity, set $R_1=R(\epsilon)$ and $R_2=R(1)=\exp(-2\pi)$, then
  \begin{eqnarray}
   L_{|\phi|}(\partial D^\epsilon_i)
    &=&\int_0^{2\pi}\sqrt{\frac{|\psi(R_1e^i\theta)|}{R_1}}R_1d\theta \nonumber\\
    &\leq&(2\pi\int_0^{2\pi}|\psi(R_1e^i\theta)|d\theta)^{1/2}\nonumber\\
    &\leq& (\frac{2\pi}{R_2-R_1}\int_{R_1}^{R_2}dr\int_0^{2\pi}|\psi(re^i\theta)|d\theta)^{1/2}\nonumber\\
    &\leq&  (\frac{2\pi}{R_2-R_1}\int_{D^*_{R_2}}|\phi|rdrd\theta)^{1/2}\nonumber\\
    &\leq&(\frac{4\pi}{R_2}||\phi||)^{1/2}.\label{eq:length1}
 \end{eqnarray}
\begin{figure}
  \includegraphics[width=100mm]{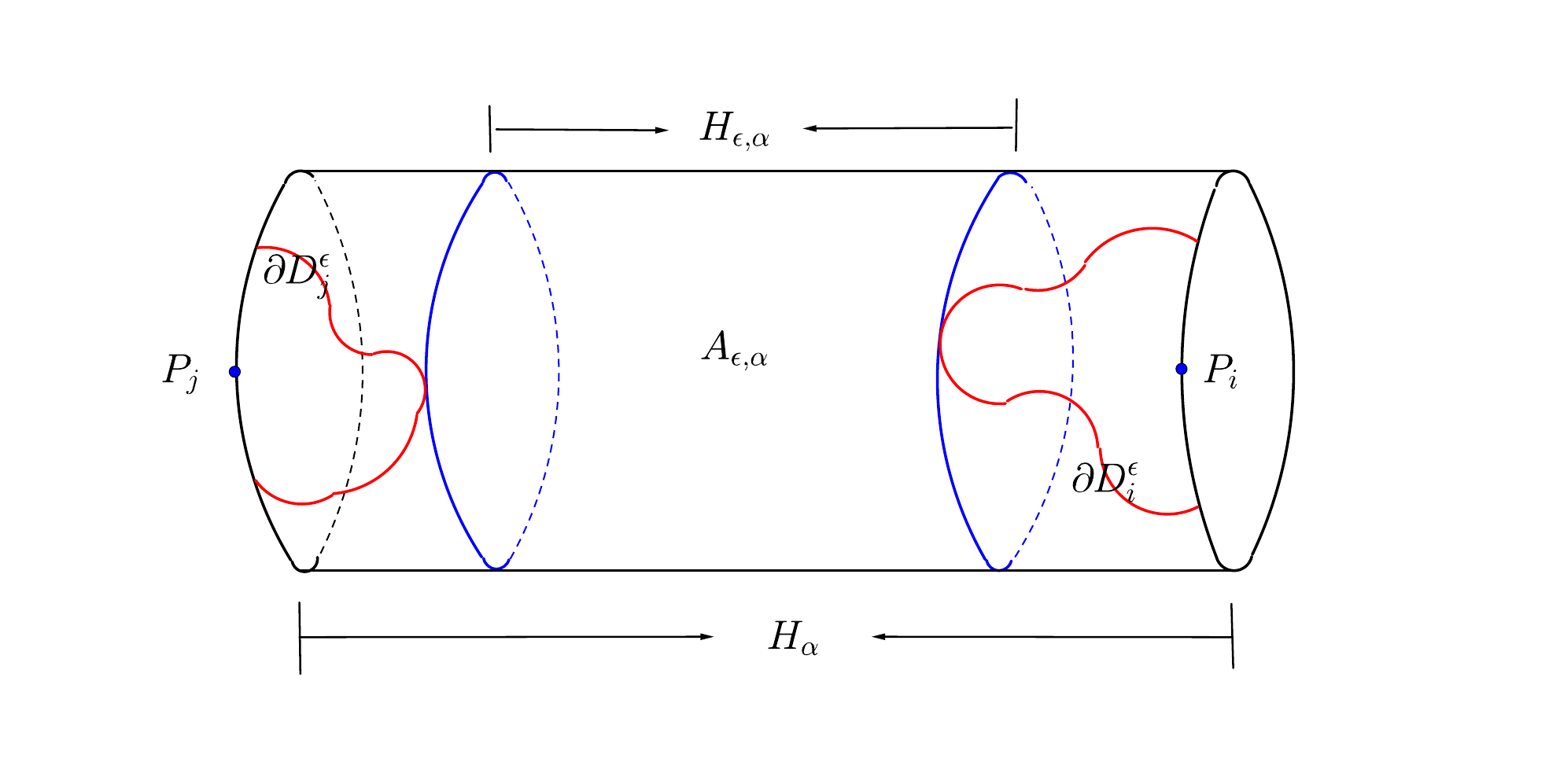}
  \caption{}
   \label{fig:cylinder}
\end{figure}
  Cutting $X$ along the critical leaves of $h_\phi$, we get a cylinder $A_\alpha$.  Let $A_{\epsilon,\alpha}\subset A_\alpha$ be the maximal cylinder whose core curve is homotopic to $\alpha$ and which is contained in $X_\epsilon$ (see Figure \ref{fig:cylinder}).  Denote by $H_\alpha$ and  $H_{\epsilon,\alpha}$ the heights of $A_\alpha$ and $A_{\epsilon,\alpha}$ respectively. Then $H_\alpha=\sqrt{(\E_X(\alpha))^{-1}||\phi||}$
  and
  $$H_{\epsilon,\alpha}\geq H_\alpha-2\max_{1\leq i\leq n}{L_{|\phi|}(\partial D^\epsilon_i)}
  \geq H_\alpha-(\frac{16\pi}{R_2}||\phi||)^{1/2}.$$
  Hence
   \[ \frac{\E_{X_\epsilon}(\alpha)}{\E_{X}(\alpha)}
    \leq\frac{\E_{A_{\epsilon,\alpha}}(\alpha)}{\E_{A_\alpha}(\alpha)}\\
    =\frac{H_\alpha}{H_{\epsilon,\alpha}}\\
    \leq 1+\sqrt{32\pi e^{2\pi}}\epsilon^{1/4}.  \]

  \textbf{Case 2}: $\E_{X_\epsilon}(\alpha)\geq \sqrt \epsilon.$   Let $\phi_\epsilon$ be the quadratic differential on $X_\epsilon$ whose horizontal measured foliation $h_{\phi_\epsilon}$ is equivalent to $\alpha$.    Denote by $|\phi_\epsilon|$ the induced flat metric on $X_\epsilon$, then $|\phi_\epsilon|$ is the extremal metric of $\alpha$. Note that
  $\partial X_\epsilon=\cup_{i=1}^n\partial D^\epsilon_i$ is contained in the critical leaves of $h_{\phi_\epsilon}$, so $\partial D^\epsilon_i$ is a geodesic under the metric $|\phi_\epsilon|$.
  Hence
  $$L_{|\phi_\epsilon|}(\partial D^\epsilon_i)\leq \sqrt{\E_{X_\epsilon}(\partial D^\epsilon_i)||\phi_\epsilon||} . $$
  Let $A_i=D^1_i\backslash D^\epsilon_i$, then $A_i$ is conformal to the annulus
   $D^*_{R(1)}\backslash D^*_{R(\epsilon)}=\{w:R(\epsilon)<|w|<R(1)\}$. Therefore
   $$\E_{X_\epsilon}(\partial D^\epsilon_i)\leq\E_{A_i}(\partial D^\epsilon_i)=2\pi(\log\frac{R(1)}{R(\epsilon)}) ^{-1}=(1/\epsilon-1)^{-1}\leq 2\epsilon.$$
 \begin{figure}
   \includegraphics[width=100mm]{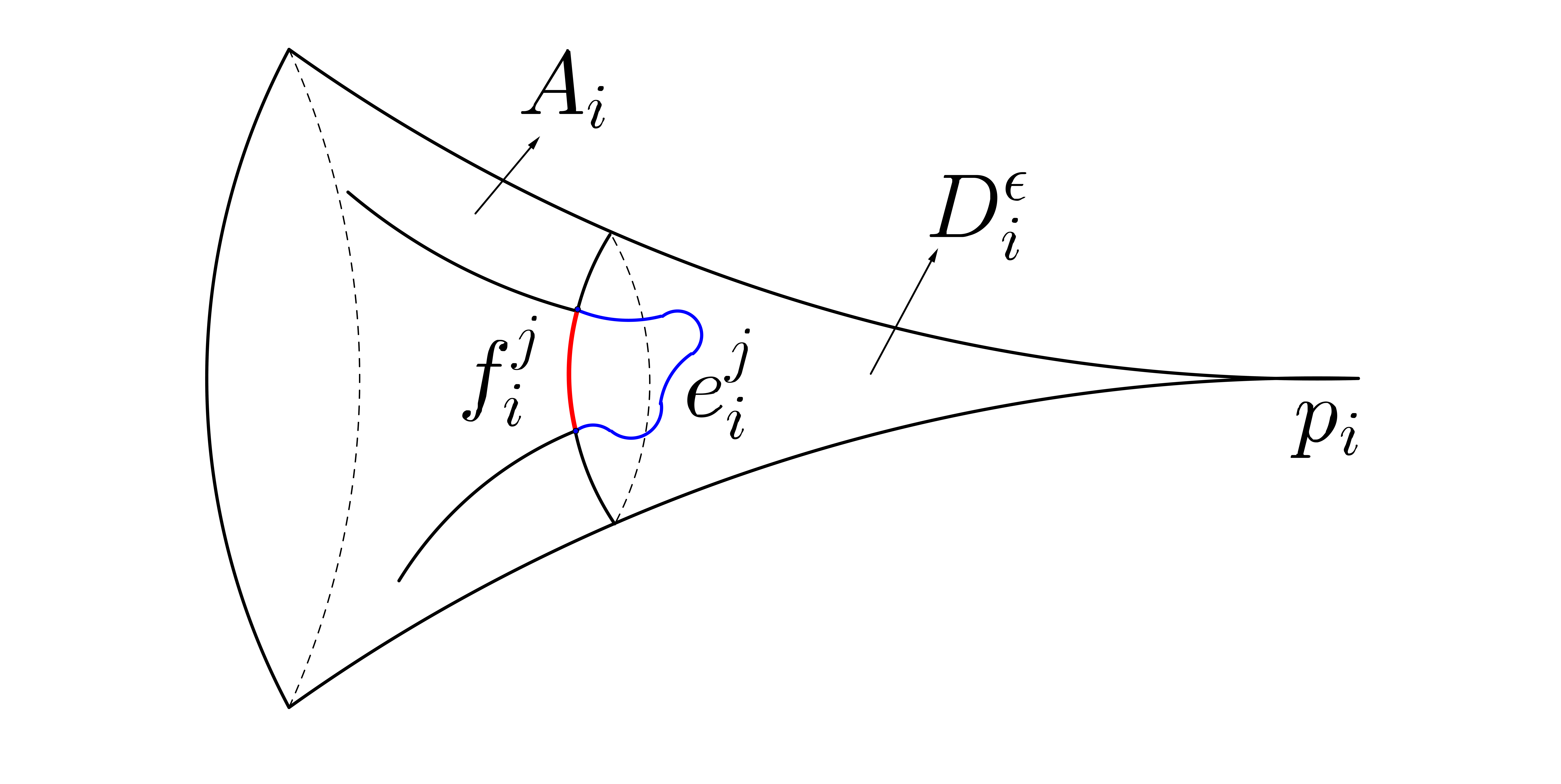}
   \caption{}
    \label{fig:cusp}
  \end{figure}
   On the other hand, $|\phi_\epsilon|$ defines a conformal metric $\rho_\epsilon$ on $X$, which coincides with $|\phi_\epsilon|$ on $X_\epsilon$ and vanishes elsewhere. For any simple closed curve $\alpha$,  set $e_i:=\alpha\cap D^\epsilon_i$. Let $e^j_{i}$ be a component of $e_i$ (see Figure \ref{fig:cusp}). $D^\epsilon_i \backslash e^j_i$ has two components, one is homeomorphic to a disc , denoted by $E^j_i$ and the other is homeomorphic to a punctured disc. Let $f_i^j=\partial E_i^j\backslash e_i^j$. It follows that $f_i^j\subset \partial D_i^\epsilon$ and $f_i^j\cap f_i^k=\emptyset$ if $j\neq k$. We construct a new simple closed curve $\alpha'$ from $\alpha$ via replacing $e_i^j$ by $f_i^j$. It is clear that $\alpha'$ is homotopic to $\alpha$ and that $\alpha'$ is contained in $X_\epsilon$. Then
   \begin{eqnarray*}
     L_{\rho_\epsilon}(\alpha)&\geq &L_{|\phi_\epsilon|}(\alpha\backslash (\cup e_i^j))\\
     &=&L_{|\phi_\epsilon|}(\alpha')-L_{|\phi_\epsilon|} (\cup f_i^j)\\
      &\geq & L_{|\phi_\epsilon|}(\alpha')-\Sigma_{i=1}^{n}L_{|\phi_\epsilon|}(\partial D^\epsilon_i)\\
      &\geq &\sqrt{\E_{X_\epsilon}(\alpha)||\phi_\epsilon||}-n\sqrt{2\epsilon||\phi_\epsilon||}\\
      &\geq& {(1-n \sqrt{2}\epsilon ^{1/4})}\sqrt{\E_{X_\epsilon}(\alpha)||\phi_\epsilon||},
    \end{eqnarray*}
    where we use $\E_{X_\epsilon}(\alpha)\geq \sqrt \epsilon$ in the last inequality.
   As a consequence,
   $$ \E_X(\alpha)\geq \frac{(l_{\rho_\epsilon}(\alpha))^2}{Area(\rho_\epsilon)}
   =\frac{( \inf_{\alpha'\in \Sim} L_{\rho_\epsilon}(\alpha))^2}{Area(\rho_\epsilon)}
   \geq (1-2n\pi \sqrt{2}\epsilon ^{1/4})^2\E_{X_\epsilon}(\alpha),$$
   where  $\alpha'\sim$ ranges over every simple closed curve homotopic to $\alpha$.

   Let $C_\epsilon=\max\{(1-2n\pi \sqrt{2}\epsilon ^{1/4})^{-2},1+\sqrt{32\pi e^{2\pi}}\epsilon^{1/4}\}$. Combining Case 1 and Case 2, we get
   $$\frac{\E_{X_\epsilon}(\alpha)}{\E_X(\alpha)}\leq C_\epsilon $$
   for any $\alpha\in\Co$. Moreover, $C_\epsilon\to1$ as $\epsilon\to 0$.
\end{proof}

\begin{corollary}
    For small $\epsilon$, there is a constant $C_\epsilon$ such that for any interior measured laminations $\mu$  and any $X\in\T_{g,n}(0)$,
  $$ {1}\leq \frac{\E_{X_\epsilon}(\mu)}{\E_{X}(\mu)}\leq {C_\epsilon}.$$
   Moreover, $C_\epsilon\to1$ as $\epsilon\to0$.
\end{corollary}

The last step of the proof is to quasiconfromally embed $X\in T_{g,n}(\epsilon)$ into $\Phi(X)\in T_{g,n}(0)$ in some nice way. We need the following theorem  due to Buser-Makover-Muetzel-Silhol (\cite{BMMS}).
\begin{theorem}[\cite{BMMS}, Theorem 2.1]\label{thm:BMMS}
  Let $l_1,l_2>0$, $0<\epsilon<1/2$, and set $\epsilon*=\frac{2}{\pi}\epsilon$. Let $Y_{l_1,l_2,\epsilon}$ be a pair of pants with boundary length $l_1,l_2,\epsilon$, and set $Y_{l_1,l_2,0}^{\epsilon*}=Y_{l_1,l_2,0}\backslash \textup{Cusp}_{\epsilon*}(Y_{l_1,l_2,0})$. Then there exists a boundary coherent (see \S \ref{ssect:FN} for the definition) quasiconformal homeomorphism
  \[ \phi: Y_{l_1,l_2,\epsilon}\rightarrow Y_{l_1,l_2,0}^{\epsilon*}\]
  with dilation $q_\phi\leq 1+2\epsilon^2$.
\end{theorem}
\remark Under the assumptions of Theorem \ref{thm:BMMS}, it is clear that there exists a boundary coherent quasiconformal homeomorphism $ \phi: Y_{l_1,\epsilon_1,\epsilon_2}\rightarrow Y_{l_1,0,,0}^{\epsilon_1*,\epsilon*_2}$ with dilation  $q_\phi\leq (1+2\epsilon_1^2)(1+2\epsilon_2^2)$.
\begin{proof}[Proof of Theorem \ref{thm:FN}]
  The second condition in Definition \ref{def} follows from the fact that $\Phi$ is a homeomorphism. It remains to verify the first condition.
  Let $X\in T_{g,n}(\epsilon)$ and $\Phi(X)\in T_{g,n}(0)$.
  It follows from Theorem \ref{thm:BMMS} that there exists a quasiconformal homeomorphism $g_1$  from $X_1$ to $X_{1,\epsilon*}:=\Phi(X_1)\backslash \textup{Cusp}_{\epsilon*}$ (resp. $g_2$  from $X_2$ to $X_{2,\epsilon*}:=\Phi(X_2)\backslash \textup{Cusp}_{\epsilon*}$ ) with dilation $K(g_i)\leq\Pi_{j=1}^n(1+2\epsilon_j^2)$, $i=1,2$. This can be obtained in the following way. Let $\{R_1,\cdots,R_{2g-2+n}\}$ be the $2g-2+n$ pairs of pants associated to the pants decomposition $\Gamma$.  If $\partial R_i\cap\partial X\neq \emptyset$, Let $h_1: R_1\to R_{1,\epsilon*}$ be the map obtained from Theorem \ref{thm:BMMS}, otherwise let $h_i: R_i\to R_i$ be the identity map. Gluing $\{h_i\}_{i=1}^{2g-2+n}$ via the Fenchel-Nielsen coordinates, we obtained the desired maps. Hence, for any interior simple closed curve $\alpha\in \Sim$,
  \[ \frac{1}{K(g_1)}\leq \frac{\E_{X_1}(\alpha)}{\E_{X_{1,\epsilon*}}(\alpha)}\leq K(g_1),\
  \frac{1}{K(g_2)}\leq \frac{\E_{X_2}(\alpha)}{\E_{X_{2,\epsilon*}}(\alpha)}\leq K(g_2).\]
  Combining with Proposition \ref{prop:ext-est-1}, we have
\[ \frac{1}{C'_\epsilon}\leq \frac{\E_{X_1}(\alpha)}{\E_{X_{1,\epsilon*}}(\alpha)}\leq C'_\epsilon C_\epsilon,\
  \frac{1}{C'_\epsilon}\leq \frac{\E_{X_2}(\alpha)}{\E_{X_{2,\epsilon*}}(\alpha)}\leq C'_\epsilon C_\epsilon,\]
where $C_\epsilon'=\Pi_{j=1}^n(1+2\epsilon_j^2)$ and $C_\epsilon$ is the constant in Proposition \ref{prop:ext-est-1}.

  Now the theorem follows from the Kerckhoff's formula on $T_{g,n}(0)$, Propositio \ref{prop:teich}  and Proposition \ref{prop:ext-est-1}.
\end{proof}

\section{Further study and questions}

\subsection{Nielsen extension}
Let $X$ be a hyperbolic surface with geodesic boundary. The \textit{infinite Nielsen extension} $X_\infty$ of $X$ is a punctured surface (see \cite{Bers}). For any $\Lambda\in\mathbb R_+^n$, we can define a map $\Psi:T_{g,n}(\Lambda)\to T_{g,n}(0)$ which associate the infinite Nielsen extension $X_\infty$ to any $X\in T_{g,n}(\Lambda)$. It is natural to ask the following question.
\begin{question}
 Given $\epsilon\in \mathbb R^n_+$,  is $\Psi: T_{g,n}(\epsilon)\to T_{g,n}(0)$ an almost isometry?
\end{question}

Unlike the Fenchel-Nielsen map $\Phi_\Gamma$, we do not know whether  $\Psi$ is a homeomorphism. But for small $\epsilon$, $\Psi$ is a surjective map. In fact, for any $X\in T_{g,n}(0)$, let $\{D^\xi_1,\cdots,D^\xi_n\}$ be the cuspidal neighbourhoods with boundary length $\xi$. For small $\epsilon$, there exists $\xi\in(0,1)^n$ such that the the boundary of $X\backslash (\cup_{1\leq k\leq n}D^\xi_k)$ has length $\epsilon$ in the  intrinsic metric. This means that $\Psi$ is a surjective map for small $\epsilon$, which means the second condition in Definition \ref{def} is satisfied.   
\begin{proposition}\label{thm:NT}
 For small $\epsilon$, $\Psi: T_{g,n}(\epsilon)\to T_{g,n}(0)$ is an almost isometry. More precisely, for $X, Y\in \T_{g,n}(\epsilon)$,
  $$ |d_T(X,Y)-d_T(X_0,Y_0) |\leq (n+3),$$
where $X_0=\Psi(X), Y_0=\Psi(Y)$.
\end{proposition}
\begin{proof}
  It follows from Proposition \ref{prop:teich} and Proposition \ref{prop:Nielsen}.
\end{proof}
\begin{proposition}\label{prop:Nielsen}
 Given $\Lambda=(\lambda_1,\cdots,\lambda_n)\in \mathbb R_+^n$. For $X\in \T_{g,n}(\Lambda)$, let $X_0\in \T_{g,n}(0)$ be the infinite Nielsen extension of $X$. There is a constant $C_\Lambda$ such that for any $\alpha\in \Sim$ and any $X\in\T_{g,n}(\Lambda)$,
  $$ {1}\leq \frac{\E_{X}(\alpha)}{\E_{X_0}(\alpha)}\leq {C_\Lambda}.$$
  Moreover, $C_\Lambda\to 0$ as $\Lambda\to0$.
\end{proposition}
\begin{proof}
 Since $X\subset X_0$, $\E_X(\alpha)\geq\E_{X_0}(\alpha)$. For the right inequality, we distinguish two cases.

   \textbf{Case 1}:  $\E_{X}(\alpha)\leq 4n^2 \lambda e^{\lambda/2}$, where $\lambda=\max_{1\leq i\leq n} \lambda_i$. By Lemma \ref{Maskit} and Proposition \ref{Halpern}, there are constants $\epsilon_0$,  $c_1,c_2$ such that if $\E_X(\alpha)<\epsilon_0$, then $c_1\leq \frac{\E_X(\alpha)}{\E_{X_0}(\alpha)}\leq c_2$ for any $\alpha\in \Sim$. If $\epsilon_0\leq\E_X(\alpha)\leq4n^2 \lambda e^{\lambda/2}$,
  then $\frac{\epsilon_0}{4n^2 \lambda e^{\lambda/2}} \leq \frac{\E_X(\alpha)}{\E_{X_0}(\alpha)}\leq\frac{4n^2 \lambda e^{\lambda/2}}{\epsilon_0}.$

   \textbf{Case 2}: $\E_{X}(\alpha)\geq 4n^2 \lambda e^{\lambda/2}$, where $\lambda=\max_{1\leq i\leq n} \lambda_i$.  Applying the method used in Case 2 in the proof  of Proposition \ref{prop:ext-est-1}, we get
   $$ \E_X(\alpha)\leq 4\E_{X_0}(\alpha).$$

    The second part follows directly from Proposition \ref{prop:ext-est-1}.
\end{proof}

\begin{proposition}[Halpern,\cite{Halpern}]\label{Halpern}
 Given $\Lambda=(\lambda_i,\cdots,\lambda_n)\in\mathbb R_+^n$ and $\lambda=\max_{1\leq i\leq n}\lambda_i$. For $X\in\T_{g,n}(\Lambda)$, let $X_\infty$ be the infinite Nielsen extension of $X$. Let $\alpha$ be a simple closed curve. If $\alpha$ is homotopic to one of the boundary components, $l_{X_\infty}(\alpha)=0$.
 Otherwise $k_\infty l_X(\alpha)< l_{X_\infty}(\alpha)<l_X(\alpha)$, where $k_\infty=\Pi_{i=1}^\infty[1-(2/\pi)\tan^{-1}(2\sinh \lambda/2^i)].$
\end{proposition}

\subsection{Improving Theorem \ref{thm:FN}}
In Theorem \ref{thm:FN}, we assume that the boundary component has small boundary length. We ask the following questions.
\begin{question}\label{ques:length}
  Does  Theorem \ref{thm:FN} still hold if we drop the condition that $\epsilon$ is small ?
\end{question}
Let $R,R'$ be two pairs of pants such that $\partial R=\{\gamma_1,\gamma_2,\gamma_3\}$ and $\partial R'=\{\gamma'_1,\gamma'_2,\gamma'_3\}$. Assume that $l(\gamma_1)=l(\gamma_1')$, $l(\gamma_2)=l(\gamma_2')$, $l(\gamma_3)=l_3$ and $l(\gamma_3')=l_3'$. One possible way to answer Question \ref{ques:length} is to find a  boundary coherent quasiconformal map $f:R\to R'$ with quasiconformal dilation only depends on $l_3,l_3'$.

\begin{question}
  Can we replace the constant $\log (n+3)$ in Theorem \ref{thm:FN} by a constant $C(\epsilon)$ such that $C(\epsilon)\to 0$  if $\epsilon\to0$?
\end{question}

\subsection{Infinite type surfaces}
A surface is of  \textit{infinite type} if it has infinite genus or infinite boundary boundary component or infinite punctures. In \cite{ALPSS}, the authors studied the Fenchel-Nielsen coordinates of the Teichm\"uller space of infinite type surfaces. In \cite{LP}, the authors studied the length spectrum metric and the Teichm\"uller metric on the Teichm\"uller space of infinite type surfaces.
\begin{question}
  Study similar questions for the Teichm\"uller space of infinite type surfaces. More formally, whether Theorem \ref{thm:arc-almost} and Theorem \ref{thm:FN} are still true if the surface in consideration has  infinite genus?
\end{question}

\end{document}